\pdfoutput=1
\documentclass[a4paper, reqno, 12pt]{amsart}
\usepackage[margin=0.8in]{geometry}
\usepackage[utf8]{inputenc}
\usepackage[T1]{fontenc}
\usepackage{lmodern, csquotes, graphicx}
\usepackage{amssymb, physics, bm}
\usepackage{hyperref}
\usepackage[subsetnonamb, bbsets]{jkmath}
\usepackage[toc,page]{appendix}
\usepackage{comment}
\usepackage{enumitem} 
\usepackage[makeroom]{cancel}
\usepackage[backend=bibtex]{biblatex}
\usepackage{dynkin-diagrams}

\newtheorem{theorem}{Theorem}[section]
\newtheorem{definition}[theorem]{Definition}
\newtheorem{lemma}[theorem]{Lemma}
\newtheorem{proposition}[theorem]{Proposition}
\newtheorem{corollary}[theorem]{Corollary}
\newtheorem{example}[theorem]{Example}
\newtheorem{notation}[theorem]{Notation}
\newtheorem{remark}[theorem]{Remark}

\newtheorem{thmintro}{Theorem}

\usepackage{tikz}
\usepackage{tikz-cd}
\usepackage{mathrsfs}
\usetikzlibrary{cd}
\usepackage{float}

\DeclareMathOperator{\Hom}{Hom}

\newcommand{\uq}{\mathbf {U}}

\newcommand{\uqb}{\mathbf{B}}

\newcommand{\uqbs}{\mathbf{B}_{\boldsymbol{c},\boldsymbol{s}}}
\newcommand{\uqds}{\mathbf{B}_{\boldsymbol{d},\boldsymbol{t}}}
\newcommand{\I}{I}

\newcommand{\Addresses}{{
		\bigskip
		\footnotesize
		\textsc{IMAPP, Radboud Universiteit, P.O. Box 9010, 6500 GL Nijmegen, The Netherlands}\par\nopagebreak
		\emph{E-mail address}: \texttt{stein.meereboer@ru.nl}
		
}}
\title{Based morphisms for characters of \mbox{quantum symmetric pairs}}
\author{Stein Meereboer}
\begin{document}
\begin{abstract}
We study based one-dimensional modules of quantum symmetric pairs over the field $\Q(q)$. We provide a complete classification of one-dimensional $\uqb$-modules that appear as submodules of simple finite-dimensional based $\uq$-modules and determine the corresponding branching rules. The main result of this paper shows that the corresponding projections are morphisms of based $\uqb$-modules. To this end we characterize one-dimensional modules at $q=\infty$, thus developing a $\imath$crystal basis theory for these modules.
This is then applied to show compatibility with the integral forms of the (dual-)canonical basis. 
\end{abstract}
\subjclass[2020]{Primary 17B37} 
\keywords{Quantum symmetric pairs, Based modules, Characters} 
	\maketitle
\section{Introduction}

\subsection{Quantum symmetric pairs}
Let $\mathfrak{g}$ be a semisimple Lie algebra over $\C$ and $\Theta:\mathfrak{g}\to \mathfrak{g}$ an involutive Lie-algebra automorphism. The fixed point subalgebra of $\Theta$ is denoted by $\mathfrak{k}=\mathfrak{g}^\Theta$. The pairs $(\mathfrak{g},\mathfrak{k})$ are known as symmetric pairs, in this paper we restrict our attention to the symmetric pairs of Hermitian type.
The theory of quantum symmetric pairs, as developed by Letzter, provides quantum analogs of the pairs of universal enveloping algebras $(U(\mathfrak{g}),U(\mathfrak{k}))$ \cite{Le99}.
Let $(\uq,\uqb)$ be a quantum symmetric pair over the field $\Q(q)$, cf. \cite{Kol14} \cite{Le99}. This pair consists of a Drinfel'd Jimbo quantum group $\uq$ and coideal subalgebra $\uqb\subset \uq$ whose structure is determined by a family of scalars $(\boldsymbol{c},\boldsymbol{s})$.

\subsection{Based $\uqb$-morphisms for the trivial module}

The theory of based $\uqb$-modules provides generalizations of Lusztig's theory of based $\uq$-modules \cite[Definition 6.11]{Bao21}. The existence of the $\imath$canonical basis \cite{Bao18} implies that every finite-dimensional based $\uq$-module is naturally equipped with the structure of a based $\uqb$-module. 

In \cite{Wat24}, Watanabe proved the \emph{stability conjecture of $\imath$canonical bases} \cite[Remark 6.18]{Bao18}. Watanabe's proof of this conjecture is a consequence of \cite[Theorem 4.3.1]{Wat24}, which verified that distinguished $\uqb$-homomorphisms to the trivial module are morphisms of based $\uqb$-modules. 
Denote $\breve{X}^+$ the subset of $X^+$ consisting of spherical weights, see Section \ref{sec:qsp}.

\begin{thmintro}\cite[Theorem 4.3.1]{Wat24}\label{thm:watintro}
Let $\lambda\in \breve{X}^+$. Then there exists a unique morphism as based $\uqb$-modules
\[L(\lambda)\to L(0),\qquad v_\lambda\mapsto 1.\]
\end{thmintro}

Theorem \ref{thm:watintro} also formed the foundation for \cite{Bao24}, where it was applied to construct integral forms of symmetric spaces over algebraically closed fields of characteristic not equal to 2. 

\subsection{Results}

The main purpose of this paper is to generalize Theorem \ref{thm:watintro} to nontrivial one-dimensional modules of $\uqb$. Here, we remind the reader that the quantum symmetric pair $(\uq,\uqb)$ is assumed to be a quantization of a symmetric pair of Hermitian type.
In \cite{Mee24}, we
classified the one-dimensional $\uqb$-modules appearing as submodules of finite-dimensional weight $\uq$-modules and determined their branching rules. However, the field of definition is the bigger field $k=\overline{\C(q)}$, the algebraic closure of $\C(q)$. This was needed to import specialization results of \cite{Let00}. It is natural to ask whether the given classification persists over the smaller field $\Q(q)\subset k$. For suitable parameters $(\boldsymbol{c},\boldsymbol{s})$ of the coideal subalgebra $\uqb$ we give an affirmative answer.
\begin{thmintro}[Proposition \ref{prop:fieldextend}
]\label{thmintro}
Let $l\in \Z$. Then there exists a character $\chi_l:\uqb\to \Q(q)$ and a unique weight $\mu_l\in X^+$ such that
	\[[L(\lambda)|_\uqb: \chi_l]=\begin{cases}
		1&\text{if\,\,}\lambda=\mu_l+\mu\text{ with }\mu\in \breve{X}^+\\
		0&\text{else}
	\end{cases}.\]
Moreover, the characters $\chi_l$ are pairwise inequivalent and provide a classification of the one-dimensional $\uqb$-modules that appear as submodules of finite-dimensional weight $\uq$-modules.
\end{thmintro}

Consequently, for each $l\in \Z$, we obtain a one-dimensional module $V_{\chi_l}=\mathrm{span}_{\Q(q)}\{1\}$ satisfying
\[b1=\chi_l(b)1,\qquad \text{where}\qquad b\in \uqb.\]

\begin{thmintro}[Theorem \ref{thm:mainthm}]\label{thm:intromain}
	Let $\lambda\in \breve{X}^+$ and let $l\in \Z$. Then there exists a unique morphism as based $\uqb$-modules
	$$L(\mu_l+\lambda) \to V_{\chi_l},\qquad \qquad  v_{\mu_l+\lambda}\mapsto 1.$$
\end{thmintro}
This generalizes \cite[Theorem 4.3.1]{Wat24} to non-trivial one-dimensional modules.
In order to prove Theorem \ref{thm:intromain} we follow the strategy outlined in  \cite{Wat24} and \cite{Wat23}. As in \cite{Wat23}, we need to understand
the one-dimensional modules at $q= \infty$. 
We characterize the  $\uqb$-modules $V_{\chi_l}$
in each highest weight $\uq$-module $L(\lambda)$ in terms of its (ordinary) crystal base. 
Following the approach in \cite{Wat24}, we reduce the general case to the rank one case. 
In the rank one case, we apply Theorem \ref{thm:watintro} to further reduce the problem to characterizing the one-dimensional modules $V_{\chi_l}$ at $q=\infty$ in its minimal realization $L(\mu_l)$. 
This final step is done by direct computation. 
The most important part in the final step is that non-trivial characters appear only in limited situations; there are just two cases to consider. Although in these cases there are infinitely many inequivalent characters, their behavior at $q=\infty$ can be deduced from the minimal cases $\chi_1$ and $\chi_{-1}$. 
Therefore, we only need to perform a small amount of case-by-case computation to adopt results from Watanabe.

We write $v=v_\lambda+\mathrm{l.o.t.}$ in the highest weight module $L(\lambda)$ to indicate that $\mathrm{l.o.t.}=\sum_{\mu<\lambda}v_\mu$ for weight vectors $v_\mu\in L(\lambda)_\mu$. Let $\mathcal A=\Z[q,q^{-1}]$ and denote $_\mathcal A L(\lambda)$ the integral form of $L(\lambda)$ arising from the integral form of $\uq$.  Let $\lambda \in \breve{X}^+$ and $l\in \Z$. By Theorem \ref{thm:intromain} there exists a vector $f^\lambda_{l}=v_\lambda+\mathrm{l.o.t.}\in L(\mu_l+\lambda)$ spanning a module of type $\chi_l$. As an application of Theorem \ref{thm:mainthm} we show that the vectors $ f^\lambda_{l}$ are compatible with integral forms.

\begin{thmintro}[Theorem \ref{lem:integral1}]
	Let $l\in \Z$. Then $f^0_{l}=v_\lambda+\mathrm{l.o.t}\in {_\mathcal AL(\mu_l)}$
\end{thmintro} 

The proof of this result is based on the following remark: If $\pi: {_\mathcal A} L(\lambda)\to {_\mathcal A V_{\chi_l}}$ is an $\mathcal A$-linear map then $\pi^\ast({_\mathcal A V_{\chi_l}}^\ast) \subset {_\mathcal A L}(\lambda)^\ast$. In the literature the $\mathcal A$-module ${_\mathcal A L}(\lambda)^\ast$ is known as the \emph{integral form of the dual canonical basis.} We can identify $L(\lambda)$ with its dual using a bilinear form. 
We denote the identification of ${_\mathcal A L}(\lambda)^\ast$ by $_\mathcal A^{\mathrm{up}}L(\lambda)$, cf. \cite[(4.2.6)]{Kas93}. This identification often leads to a strict inclusion $_\mathcal A L(\lambda)\subsetneq {_\mathcal A^{\mathrm{up}}L(\lambda)}$. However, when the weight $\lambda$ is (quasi-)minuscule, we have an equality (up to the zero weight space), cf. \cite[\S 9.21]{Jan96}. This equality allows us to identify the \emph{minimal} spherical vectors in the integral form, which can be extended to all characters using tensor products.
\begin{center}
	$\textsc{Acknowledgment}$
\end{center}
The paper was inspired by conversations with Huanchen Bao during my visit to the NUS. I would to like to Huanchen Bao for the
discussions and the hospitality. Moreover, I would like to thank Erik Koelink for corrections and valuable comments.
This work is funded by grant number \texttt{OCENW.M20.108} of the
Dutch Research Council. 
\section{Quantum groups and and based modules}\label{Sec:QG}
\subsection{Quantum groups}
We introduce notation of quantum groups as in \cite{Lus10}. Let $\mathfrak{g}$ be a complex semisimple Lie algebra with  Cartan matrix $C=(c_{ij})_{i,j\in \I}$. Let $D=\operatorname{diag}(d_i\,:\,d_i\in \Z_{\geq 1}, \,i\in\I)$ be a symmetrizer, meaning that $DC$ is symmetric and $\gcd\{d_i\,:\, i\in\I\}=1$. We fix a set of simple roots $\Pi=\{\alpha_i\,:\,i\in \I\}$ and a set of simple coroots $\Pi^\vee=\{\alpha^\vee_i\,:\,i\in \I\}$ with corresponding root and coroot systems $\mathcal R$ and $\mathcal R^\vee$. The root lattice is denoted by $\Z\I=\oplus_{i\in \I } \Z\alpha_i$. The root lattice is equipped with the normalized Killing form $(\,,\,)$, scaled so that short roots have squared length $2$. The Weyl group $W$ is generated by the simple reflections $s_i:\Z\I\to \Z\I$, for $i\in \I$ defined by $s_i(\alpha_j)=\alpha_j-c_{ij}\alpha_i,$ for $j\in \I$. 

Let $q$ be an indeterminate, and let $\Q(q)$ denote the field of rational functions in $q$ with coefficients in $\Q$. Let  $\overline{\,\cdot \,}:\Q(q)\to \Q(q)$ be the unique $\Q$-linear automorphism such that $q$ is mapped to $q^{-1}$.  For each $i\in\I$ and $n\in \Z_{\geq0}$, set
\[q_i:= q^{d_i},\qquad [n]_i:=\cfrac{q_i^n-q_i^{-n}}{q_i-q_i^{-1}}\qquad\text{and}\qquad [n]_{i}!:=\prod_{k=1}^n[k]_i.\]
Recall $\mathcal A=\Z[q,q^{-1}]$ and let $\mathbf{A}_\infty=\Q(q)\cap \Q[[q^{-1}]]$, the latter consists of all rational functions regular at $q= \infty$.

Let $(Y,X, \langle \,,\, \rangle.\dots)$ be a root datum of type $(\I,\, (\,,\,))$, cf. \cite[\mbox{2.2}]{Lus10}.  By definition, there are embeddings $\I\to X$, $i\mapsto \alpha_i$ and $\I\to Y$, $i\mapsto \alpha_i^\vee$.  The pairing $\langle\,,\,\rangle$ and the form $(\,,\,)$ are related by the formula $\langle \alpha_i^\vee,\alpha_j\rangle=\frac{2(\alpha_i,\alpha_j)}{(\alpha_i,\alpha_i)}$, for $i,j\in \I$. The lattice $X$ is partially ordered by the rule
$$\lambda\leq\lambda'\quad\text{if and only if}\quad \lambda'-\lambda\in \Z_{\geq0}\I.$$
\begin{definition}[Drinfel'd Jimbo quantum group]
The \emph{quantum group} $\uq$, associated to the root datum, is the unital associative algebra over $\Q(q)$, generated by the symbols $E_i,F_i$ for $i\in \I$ and $K_{h}$ for $h\in Y$. These generators satisfy the following relations for $ i,j\in\I,$ and $ h,h_1,h_2\in Y$
\begin{align*}
	K_0&=1,\\
	K_{h_1}K_{h_2}&=K_{h_1+h_2},\\
	K_{h}E_i&= q^{\langle h,\alpha_i\rangle}E_iK_h,\\
	K_{h}F_i&= q^{-\langle h,\alpha_i\rangle}F_iK_h,\\
	E_iF_j-F_jE_i&=\delta_{i,j}\cfrac{K_i-K_i^{-1}}{q_i-q_i^{-1}},
\end{align*}
as well as the quantum Serre relations for $i\neq j$,

\begin{align*}
	\sum_{r+s=1-a_{i,j}}(-1)^sE_i^{(r)}E_jE_i^{(s)}=\sum_{r+s=1-a_{i,j}}(-1)^s F_i^{(r)}F_jF_i^{(s)}=0.
\end{align*}
Here we use the notation
$$K_i:=K_{d_i\alpha_i^\vee},\qquad E_i^{(n)}:=\frac{1}{[n]_i!}E_i^n\qquad\text{and}\qquad F_i^{(n)}=\frac{1}{[n]_i!}F_i^n.$$
\end{definition}
The quantum group $\uq$ has the structure of a Hopf-algebra $(\uq, \triangle,\epsilon,\iota,S)$, where we take the conventions of the Hopf-algebra structure as given in \cite[\mbox{Ch 4}]{Jan96}. 
We use \emph{Sweedler notation} to denote the coproduct, writing $\triangle(x)=x_{(1)}\otimes x_{(2)}\in \uq\otimes \uq$. The \emph{bar involution} $\overline{\,\cdot\,}: \uq\to\uq$ is the unique $\Q$-linear automorphism with
\[E_i\mapsto E_i,\quad F_i\mapsto F_i,\quad K_i\mapsto K_i^{-1},\quad q\mapsto q^{-1},\quad \text{for all }i\in I,\]
this is well-defined by \cite[\S3.1.12]{Lus10}.
Set $_\mathcal A\uq$ equal to the $\mathcal A$-subalgebra generated $K_h$ with $h\in Y$ and by all divided powers 
$F^{(a)}_i = F^a_i /[a]_i!$, $E^{(a)}_i = E^a_i /[a]_i!$ for $a\geq 0$ and $i\in I$.
\subsection{Based modules}\label{sec:base}

A $\uq$-module $M$ is called a \emph{weight} $\uq$-module if
\[M=\bigoplus_{\lambda\in X}M_{\lambda},\qquad \text{where}\qquad M_\lambda=\{m\in M\,:\, K_h m=q^{\langle h,\lambda\rangle}m\,\text{for all }h\in Y\}.\]
Moreover, we say that a weight module is \emph{integrable} if for each $i\in I$, the Chevalley generators $E_i$ and $F_i$ act locally nilpotent on $M$.
The simple left highest weight $\uq$-module corresponding to a dominant weight $\lambda\in X^+$ is denoted by $L(\lambda)$ cf. \cite[\mbox{3.5}]{Lus10}. We write $v_\lambda$ for the highest weight vector and $B(\lambda)$ for the corresponding canonical basis, cf. \cite[\S 14.4.12]{Lus10}.
\begin{definition}[Based module]
	A \emph{based module} is a pair $(M,B)$, consisting of a weight $\uq$-module $M$ and a basis $B$ of $M$ satisfying
	\begin{enumerate}
		\item For each $\lambda \in X$ the set $B\cap M_\lambda$ is a basis of $M_\lambda$.
		\item The $\mathcal A$-submodule $_\mathcal AM$ of  $M$ spanned by $B$ is stable under the action of $_\mathcal A\uq $.
		\item The $\Q$-linear endomorphism $\overline{\,\cdot\,}$ on $M$ , defined by $\overline{q^nb}= q^{-n}b$ for all $n \in \Z$ and
		$b \in B$, is compatible with the bar involution on $\uq$:
		\[\overline{xm}=\overline{x}\overline{m},\qquad \text{where }\qquad x\in \uq,m\in M.\]
		\item Set $\mathcal L := \mathbf{A}_\infty B$ and $\mathcal B := \{b + (q^{-1})\mathcal L \,: \,b \in B\}$. Then, the set $\mathcal B$ is a $\Q$-basis of $\mathcal L/q^{-1}\mathcal L$.
	\end{enumerate}
\end{definition}

\begin{example}
For each $\lambda\in X^+$ the pair $(L(\lambda),B(\lambda))$ is based, cf. \cite[\S 27.1.4]{Lus10}. 
\end{example}

\begin{definition}[Based submodules and based morphisms]\label{def:basedmor}Let $(M,B)$ be a based module. A $\uq$-submodule $N\subset M$ is called a \emph{based submodule} if $N$ is spanned by a subset $B_N\subset B$. 
We remark that $(N,B_N)$ is a based $\uq$-module. In this case, the quotient $(M/N, B\setminus B_N+N)$ is a based $\uq$-module. 

Let $(M_1,B_1)$ and $(M_2,B_2)$ be based $\uq$-modules and $f: M_1\to M_2$ a homomorphism of $\uq$-modules. Then $f$ is called a morphism of based $\uq$-modules if $f(B_1)\subset B_2\cup \{0\}$ and $\ker (f)$ is a based submodule of $M_1$.
\end{definition}

\begin{example}\label{ex:tensor}
Let $\lambda,\mu \in X^+$. By \cite[Theorem 27.3.2]{Lus10}, the canonical basis \[B=\{ b_1\diamondsuit b_2 : b_1\in B(\lambda), b_2\in B(\mu)\}\]
of $ L(\lambda)\otimes L(\mu)$ satisfies a triangularity condition;  
\begin{equation}\label{eq:triangular}
	b_1\diamondsuit b_2=b_1\otimes b_2+\sum _{(b_1, b_2)>(b_1', b_2')}a_{b_1', b_2'}b_1'\otimes b_2'\qquad \text{for scalars}\qquad a_{b_1', b_2'}\in q^{-1}\Z[q^{-1}] .
\end{equation}
Here $>$ denotes a partial ordering on the Cartesian product of the canonical basis, cf. \cite[\S 27.3]{Lus10}. 
\end{example}

\begin{notation}\label{not:cartan}
Let $\lambda,\mu\in X^+$, then we denote $\pi_{\lambda+\mu}: L(\lambda)\otimes L(\mu)\to L(\lambda+\mu)$ the $\uq$-module morphism known as the Cartan projection.
\end{notation}

\begin{notation}\label{not:infty}
Let $v,w\in \mathcal L$, then we write $v\equiv_\infty w$ if $v-w\in q^{-1}\mathcal L$.
\end{notation}

\begin{lemma}\label{lem:preserveaform}
Let $\lambda,\mu \in X^+$
and let $v\in  {\mathcal L}(\lambda)$, $w\in {\mathcal L}(\mu)$ with $v\equiv_{\infty} v_\lambda$, $w\equiv_{\infty} v_\mu$. 
Then $\pi_{\lambda+\mu}(v\otimes w)\equiv_{\infty} v_\lambda\otimes v_\mu$.
\end{lemma}

\begin{proof}
We recall the canonical basis $B=\{ b_1\diamondsuit b_2 : b_1\in B(\lambda), b_2\in B(\mu)\}$ of $M=L(\lambda)\otimes L(\mu)$ introduced in Example \ref{ex:tensor}.
We note that
\[v\otimes w-v_\lambda\otimes v_{\mu}\in q^{-1}\mathcal L(\lambda)\otimes \mathcal L(\mu).\]
By triangularity \eqref{eq:triangular} of the canonical base of $M$ and \cite[Proposition 27.1.7]{Lus10}, it follows that $\pi_{\lambda+\mu}(v\otimes w)\equiv_\infty v_\lambda\otimes v_{\mu}$.
\end{proof}

The dual canonical basis of Lusztig, also known as Kashiwara's upper global basis, has several constructions and interpretations. In this section we introduce the dual canonical basis of highest weight modules.
Set $\varrho:\uq\to \uq$ equal to the unique involutive algebra anti-automorphism with
\begin{equation}\label{eq:varrho}
E_i\mapsto  q_iK_iF_i, ,\qquad F_i\mapsto q_iK^{-1}_i E_i,\qquad K_h\mapsto K_h,\qquad \text{where}\qquad i\in I, h\in Y.\end{equation}
Let $\lambda \in X^+$. By \cite[Proposition 19.1.2]{Lus10}, there exists a unique non-degenerate symmetric $\varrho$-contravariant bilinear form $(\,,\,):L(\lambda)\times L(\lambda)\to \Q(q)$ with $(v_\lambda,v_\lambda)=1$. For a left $\uqb$-module $M$, we consider the dual $M^\ast=\Hom(M,\Q(q))$ as a right $\uqb$-module via
\[[f\triangleleft x](v)=f(xv),\qquad \text{where}\qquad f\in M^\ast ,x\in \uq, v\in M.\]
The preceding bilinear form induces an isomorphism of right $\uq$-modules $L(\lambda)^\ast \cong {^\varrho L(\lambda)}$ via
\begin{equation}\label{eq:vectiso}
v\mapsto (w\mapsto (v,w)),\qquad \text{where}\qquad v,w\in L(\lambda).
\end{equation}
Here, $^\varrho L(\lambda)$  denotes the right $\uq$-module with underlying vector space $L(\lambda)$ where the action is twisted by $\varrho$, cf. \cite[Definition 2.2]{Mee24}
\begin{definition}[Dual canonical basis]\label{def:dualcan}
Let $(M,B)$ be a based $\uq$-module. Then, the dual canonical basis $B^{\ast}=\{b^\ast:b\in B\}$ of $M^\ast$ is the dual of $B$; that is to say
\[b^\ast(b')=\delta_{b,b'},\qquad \text{where}\qquad b,b'\in B.\]
\end{definition}

Using the vector space isomorphism $L(\lambda)^\ast \cong L(\lambda)$ from \eqref{eq:vectiso}, this defines an integral form $_\mathcal A L^{\mathrm{up}}(\lambda)={\mathrm{span}_{\mathcal A}}B^{\ast}(\lambda)$ dual to $_\mathcal A L(\lambda)$, in the sense that
\[_\mathcal AL^{\mathrm{up}}(\lambda)=\{v\in L(\lambda): (v, {_\mathcal AL(\lambda))}\subset \mathcal A\}.\]
According to \cite[Proposition 19.3.3]{Lus10}, we have $(b,b')\in \mathcal A$. Thus, $_\mathcal AL^{\mathrm{up}}(\lambda)$ is a $_\mathcal A\uq$-module with $_\mathcal AL^{\mathrm{up}}(\lambda)\supset _\mathcal AL(\lambda)$. 
\section{Quantum symmetric pairs and based $\uqb$-modules}\label{Sec:QSP}
In the following we introduce quantum symmetric pairs in the sense of \cite{Kol14} and \cite{Le99}. 
\subsection{Admissible pairs and $\imath$root datum}
\begin{notation}
	For each $I_\bullet\subset I$ denote $W_{\bullet}\subset W$ the parabolic subgroup associated with $I_\bullet$, i.e the subgroup generated by the reflections $s_i$ with $i\in I_\bullet$. Let $w_\bullet$ be the longest element of the parabolic subgroup $W_{\bullet}$. 
\end{notation}

\begin{definition}[Admissible pair]
An \emph{admissible pair} $(\I_\bullet,\tau)$ consists of a subset $\I_\bullet\subset \I$ and an involution $\tau:\I\to \I$ with
\begin{enumerate}
	\item $a_{i,j}=a_{\tau i, \tau j}$ for all $i,j \in I$.
	\item $\tau|_{\I_\bullet}=-w_\bullet$.
	\item If $i\in I_\circ:=\I\setminus\I_\bullet$ and $\tau i=i$ then $\langle \rho_\bullet^\vee,\alpha_i\rangle\in \Z$.
\end{enumerate}
Here $\rho_\bullet^\vee$ denotes the half sum of positive coroots relative to $\I_\bullet$.
\end{definition}
We only consider admissible pairs that are \emph{irreducible}, i.e. for each $i,j \in I$, there exists
a sequence $i= i_1\dots ,i_r = j \in I$ such that for each $k = 1,...,r- 1$, we have either
$a_{i_k ,i_{k+1}} \neq 0$ or $i_{k+1} = \tau i_k$. 

\begin{definition}[Rank, rank one admissible pair]\label{def:rankone}	The \emph{rank} of an admissible pair $(I_\bullet,\tau)$ equals the number of $\tau$-orbits in $I_\circ$. Each $i\in I_\circ$ induces an irreducible admissible pair $( I_\bullet,\tau |_{I_{i\cup \tau i}\cup I_\bullet})$ of \emph{rank one}; if needed we restrict it to the irreducible component containing $i$.
\end{definition}

Let
 $\theta = -w_\bullet \circ \tau$. It defines an involution the weight lattice $X$ and the coroot lattice $Y$.
We fix the following notation as in \cite[(3.3)]{Bao18}
\begin{equation*}
	X_\imath=X/\langle \lambda-\theta\lambda\mid\lambda\in X\rangle\qquad \text{and}\qquad Y^\imath=\{h\in Y\mid \theta h=h\}.
\end{equation*}
The abelian group $X_\imath$ is called \emph{$\imath$weight lattice}, and the abelian group $Y^\imath$ the \emph{$\imath$coweight lattice}. We denote $\overline{\,\cdot\,}: X\to X_\imath$ the quotient map.
The pairing between $Y$ and $X$ induces a pairing between $Y^\imath\times X_\imath\rightarrow\mathbb{Z}$ defined by
\[\langle \overline{\lambda},h\rangle:= \langle \lambda,h\rangle,\qquad \text{where}\qquad \lambda \in X,h\in Y^\imath. \] 
Define $\breve{X}^+=\{\lambda\in X^+\,:\,\overline{\lambda}=0\}$, we refer to their elements as \emph{spherical weights.} The \emph{restricted root system} $\Sigma$ associated to the involution $\Theta$ equals

\begin{equation}\label{eq:restroot}
	\Sigma=\big\{\frac{\alpha-\Theta \alpha}{2}:\alpha\in \mathcal R\big\}.\end{equation}
Let \[I_{\mathrm{ns}}:=\{i\in I_\circ\,:\, \tau i=i,\,\langle \alpha_i^\vee,\alpha_j\rangle =0\text{ for all }\}.\]
\subsection{Quantum symmetric pairs}\label{sec:qsp}
For each $i\in I_\circ $, consider a family of \emph{parameters} $(\boldsymbol{c},\boldsymbol{s})=(c_i,s_i)_{i\in I_\circ}$ with $c_i \in \mathcal A^\times$ and $s_i\in \mathcal A$ for all $i\in I_\circ$.

\begin{remark}\label{rem:par}
	The parameters $(c_i,s_i)_{i\in I_\circ}$ are assumed to coincide with those of \cite[Lemma 4.2.1]{Wat23}. That is to say
	$s_i=0$ and 
	\[c_i=
	\begin{cases}
		q^{-1} & \text{if } ( I_\bullet,\tau |_{I_{i\cup \tau i}\cup I_\bullet,}) \text{ is of type }\mathsf{AIII_b^1}, \\
		q & \text{if } ( I_\bullet,\tau |_{I_{i\cup \tau i}\cup I_\bullet,}) \text{ is of type }\mathsf{AII_3}, \\
		1 & \text{if } ( I_\bullet,\tau |_{I_{i\cup \tau i}\cup I_\bullet,})\text{ is of type }\mathsf{AIII_2}, \\
		q^{2n-3} & \text{if } ( I_\bullet,\tau |_{I_{i\cup \tau i}\cup I_\bullet,}) \text{ is of type }\mathsf{BII_n}, \\
		q^{n-1} & \text{if } ( I_\bullet,\tau |_{I_{i\cup \tau i}\cup I_\bullet,}) \text{ is of type }\mathsf{CII_n}, \\
		q^{n-2} & \text{if } ( I_\bullet,\tau |_{I_{i\cup \tau i}\cup I_\bullet,}) \text{ is of type }\mathsf{DII_n}, \\
		q^5 & \text{if } ( I_\bullet,\tau |_{I_{i\cup \tau i}\cup I_\bullet,}) \text{ is of type }\mathsf{F II_4}.
	\end{cases}
	\]	
	 with the following exceptions: 
	\begin{enumerate}\label{enum:2}
		\item For $i\in I_\circ$ with rank one diagram of type $\mathsf{AIV^m}=\mathsf{AIII_a}$ (see Table \ref{table:1}, with $m=|I_\bullet|$), we allow all parameters $c_i,c_{\tau i}$ that satisfy $c_ic_{\tau i}=(-1)^m q^{m-1}$. They depend on an integer $n$ by $s=0$, $c_{1,n}=q^n$ and $c_{m,n}=(-1)^mq^{m-1-n}$.
		\item For $i\in I_\circ$ with rank one diagram of type $\mathsf{AI}=\mathsf{AIII_b}$ (see Table \ref{table:1}) and $\langle \alpha_k^\vee,\alpha_i\rangle\in 2\Z$ for all $k\in I_{\mathrm{ns}}$, we allow $s_i=[n]_i$ with $n\in\Z$.
	\end{enumerate}
\end{remark}

\begin{definition}[QSP coideal subalgebra]\label{def:qsp}
Let $\uqb=\uqbs \subset \uq$ be the \emph{quantum symmetric pair coideal subalgebra} defined in \cite{Kol14}. It is the $\mathbb{Q}(q)$-subalgebra of $\uq$ generated by elements
\[
B_i=F_i+c_i T_{w_{\bullet}}(E_{\tau i})K_i^{-1}+s_iK_i^{-1},\quad F_j, \quad E_j,\quad K_h, \quad\text{where}\quad i\in I_\circ ,j\in I_\bullet, h\in Y^\imath,\]
were $T_{w_{\bullet}}=T_{w_{\bullet},1}''$ is the Lusztig's braid group operator associated with $w_\bullet$, cf. \cite[\S 37.1]{Lus10}. The pairs $(\uq, \uqb)$ are called \emph{quantum symmetric pairs.} 
\end{definition}
The term \emph{coideal subalgebra} in Definition \ref{def:qsp} refers to the fact that $\triangle(\uqb)\subset \uqb\otimes \uq, $ cf. \cite[Proposition 5.2.]{Kol14}. 
The \emph{type} of a quantum symmetric symmetric pair refers to the type of the admissible pair $(I_\bullet,\tau)$. These are the Satake diagrams classified in \cite{Araki}. The \emph{rank} of the quantum symmetric pair $(\uq,\uqb)$ refers to the rank of its corresponding admissible pair.

Recall from Definition \ref{def:rankone} that each $i\in I_\circ$ induces an irreducible admissible pair of rank one.

\begin{definition}[Rank one QSP]\label{def:rank}
Let $i\in I_\circ$. Then, we denote $(\uq_i,\uqb_i)$ the rank one quantum symmetric pair associated to the rank one admissible pair $( I_\bullet,\tau |_{I_{i\cup \tau i}\cup I_\bullet,})$.
\end{definition}

We denote $\uq_\bullet\subset\uqb$ the $\Q(q)$-subalgebra generated by
\[
F_j, \quad E_j,\quad K_j, \qquad\text{where}\qquad j\in I_\bullet.\]
A consequence of the assumptions on the parameters is that the \emph{$\imath$bar involution }exists on $\uqb$, cf. \cite[Lemma 3.15]{Bao18}. 
\begin{definition}[$\imath$bar involution]
The \emph{$\imath$bar involution} is the unique $\Q$-algebra involution $\psi^\imath:\uqb\to\uqb$ with
\[B_i\mapsto B_i,\qquad K_h\mapsto K_{-h},\qquad  \psi^\imath|_{\uq_\bullet}=\overline{\,\cdot\,}|_{\uq_\bullet},\qquad q\mapsto q^{-1},\qquad\text{where}\qquad i\in I_\circ, h\in Y^\imath.\]
\end{definition}

\subsection{Based $\uqb$-modules}\label{subsec:based}
In analogy of based $\uq$-modules, we introduce the notion of a based $\uqb$-module, cf. \cite[\S6.3.1.]{Bao21}.
\begin{definition}\cite[\S3.2]{Wat2025}
A $\uqb$-module $M$ is called a \emph{weight $\uqb$-module} if
\[M=\bigoplus_{\lambda\in X_\imath}M_{\lambda},\qquad \text{where}\qquad M_\lambda\subset\{m\in M\,:\, K_h m=q^{\langle h,\lambda\rangle}m\,\text{ for all }h\in Y^\imath\}.\]
such that 
\[ E_j M_\zeta \subset M_{\zeta+\overline{\alpha_j}},\qquad F_j M_\zeta \subset M_{\zeta-\overline{\alpha_j}},\qquad B_kM_\zeta \subset  M_{\zeta-\overline{\alpha_k}},\qquad \text{for all}\qquad j\in I_\bullet,k\in I_\circ.\]
\end{definition}

\begin{example}
Each weight $\uq$-module $M$ is a weight $\uqb$-module by restriction, to be precise
\begin{equation}\label{eq:iweightspace}
	M=\bigoplus_{\lambda\in X_\imath}M_{\lambda},\qquad \text{where}\qquad M_\lambda=\bigoplus_{\mu\in X\, \overline{\mu}=\lambda} M_\mu.
\end{equation}
\end{example}

Next, we recall some facts about the modified form of $\uqb$. Its definition is needed to introduce the notion of based $\uqb$-modules.

Let $\dot{\uqb}$ be the modified algebra of $\uqb$ with idempotents $\{\boldsymbol{1}_\zeta\,:\,  \zeta \in X_\imath\}$, we refer the reader for its definition to \cite[\S 3.5]{Bao18}.
By construction each weight $\uqb$-module is a $\dot{\uqb}$-module, where the idempotents $1_\zeta$ act on weight $\uqb$-modules by projection on the corresponding weight space.

Let $i\neq \tau i=w_\bullet i$. Then, the parity of $\langle \alpha_i^\vee,\lambda\rangle\in \Z$ is independent of representative of $\overline{\lambda}$, we denote $p_i (\overline{\lambda})$ the parity in $\Z_2=\{\overline{0},\overline{1}\}$, cf. \cite[\S 4.1.1]{Wat2025}, we use this notation in Definition \ref{def:idiv}.
\begin{definition}[$\imath$divided power]\cite[Theorem 5.1]{Bao21}\cite[\S 4.1]{Wat2025}\label{def:idiv}
Let $i\in I_\circ$, $a\in \Z_{\geq0}$ and $\lambda \in X$. We denote by $B_{i,\overline{\lambda}}^{(a)}\in \dot{\uqb}$ the \emph{$\imath$divided power}, they are given by
\begin{enumerate}\label{enum:divided}
	\item 
	Let $\alpha_i\neq \tau \alpha_i$ or $\alpha_i\neq w_\bullet \alpha_i$. Then 
	$	B_{i,\overline{\lambda}}^{(a)}=\frac{B_i^a}{[a]_i!}\boldsymbol{1}_{\overline{\lambda}}.
$
	\item Let $\alpha_i=\tau \alpha_i\neq w_\bullet \alpha_i$. 
	Then the elements $B_{k,\overline{\lambda}}^{(a)}\in \dot{\uqb}$ are inductively defined according to
	\begin{enumerate}
		\item $B_{k, \overline{\lambda}}^{(0)} := \boldsymbol{1}_\zeta$.
		\item If $p_i(\overline{\lambda}) = \bar{0}$, then $B_{k,\zeta}^{(1)} := (B_k - s_i)\boldsymbol{1}_{\overline{\lambda}}$.
		\item If $p_i(\overline{\lambda}) = \overline{a}$ and $a \geq 1$, then $B_{i,\overline{\lambda}}^{(a)} := \frac{1}{[a]_i} B_i B_{i,\overline{\lambda}}^{(a-1)}$.
		\item If $p_i(\overline{\lambda}) \neq \overline{a}$ and $a \geq 2$, then
		\[
		B_{i, \overline{\lambda}}^{(a)} := \frac{1}{[a]_i[a-1]_i}(B_i^2 - (q_i^{a-1} + q_i^{-a+1})s_i B_i + (s_i^2 - q_ic_i[a-1]_i^2 )) B_{i,\overline{\lambda}}^{(a-2)}.
		\]
	\end{enumerate}
\end{enumerate}
\end{definition}
\begin{definition}
Let $_\mathcal A\dot{\uqb}$ be the $\mathcal A$-form of $\dot{\uqb}$ defined by
\begin{equation}\label{eq:dotaform}
_\mathcal A\dot{\uqb}=\{x\in \dot{\uqb}\,: x{_\mathcal AM}\subset {_\mathcal AM}\text{ for all integrable }\uq\text{-modules }M\,\}.
\end{equation}

\end{definition}
By \cite[Corollary 7.5]{Bao18}, the $\mathcal A$-algebra $_\mathcal A\dot{\uqb}$ is generated by 
\begin{equation}\label{eq:genaform}
B_{i,\zeta}^{(a)},\qquad \text{and}\qquad E_j \boldsymbol{1}_{\zeta}\qquad\text{ where}\qquad \zeta\in X_\imath, i\in I_\circ , j\in I_\bullet \,\text{ and }\,a\geq0.
\end{equation}

\begin{definition}[Based $\uqb$-module]\cite[Definition 6.11]{Bao21}\label{def:Bbased} 
A \emph{based $\uqb$-module} is a pair $(M,B)$, consisting of a weight $\uqb$-module $M$ and a basis $B$ of $M$ satisfying:
\begin{enumerate}
	\item For each $\lambda \in X_\imath$ the set $B\cap M_\lambda$ is a basis of $M_\lambda$.
	\item The $\mathcal A$-submodule $_\mathcal AM$ of  $M$ spanned by $B$ is stable under the action of $_\mathcal A \dot{\uqb}$.
	\item The $\Q$-linear endomorphism $\psi^\imath$ on $M$, defined by $\psi^\imath(q^nb)= q^{-n}b$ for all $n \in \Z$ and
	$b \in B$, is compatible with the bar involution on $\uqb$;
	\[\psi^\imath(xm)=\psi^\imath(x)\psi^\imath(m),\qquad \text{where}\qquad x\in \uqb,m\in M.\]
	\item Set $\mathcal L := \mathbf{A}_\infty B$ and $\mathcal B := \{b + q^{-1}\mathcal L \,: \,b \in B\}$. Then, the set $\mathcal B$ is a $\Q$-basis of $\mathcal L/q^{-1}\mathcal L$.
\end{enumerate}

\end{definition}

In Example \ref{ex:BW} we recall that each finite dimensional based $\uq$-module is naturally equipped with the structure of a based $\uqb$-module.

\begin{example}\cite[Theorem 6.12]{Bao21}, \cite[\S 2.2.2]{Bao24}\label{ex:BW}
Let $\lambda\in X^+$ and let $\Upsilon: L(\lambda)\to L(\lambda)$ denote the \emph{quasi-K} matrix, cf. \cite[\S6]{Ko19}. Then the involution $\psi^\imath=\Upsilon\circ \overline{\,\cdot\,}: L(\lambda)\to L(\lambda)$ is compatible with the bar involution on $\uqb$, that is to say
	\[\psi^\imath(xv)=\psi^\imath(x)\psi^\imath(v),\qquad \text{where}\qquad x\in \uqb,v\in L(\lambda).\]
	Moreover, $L(\lambda)$ admits a unique basis $B^\imath(\lambda)=\{b^\imath: b\in B(\lambda)\}$ such that for all $b\in B(\lambda)$:
	\begin{enumerate}
		\item $\psi^\imath(b^\imath)=b^\imath$.
		\item $b^\imath=b+\sum_{\mathrm{wt}(b')<\mathrm{wt}(b)}a_{b,b'}b'$, with $a_{b,b'}\in q^{-1}\Z[q^{-1}]$ for all $b'\in B(\lambda)$.
		\item The $\mathcal A$ and $\mathbf{A}_{\infty}$-submodules spanned by $B^\imath$ equal $_\mathcal A L(\lambda)$ and $\mathcal L(\lambda)$, respectively.
		\item $(L(\lambda),B^\imath(\lambda))$ is a based $\uqb$-module.
	\end{enumerate}
	
\end{example}

We write $\mathcal B= B+q^{-1}\mathcal L$. In analogy to the case of $\uq$-modules (Definition \ref{def:basedmor}), one defines the notion of a based $\uqb$-submodule and morphisms of based $\uqb$-modules. 

\begin{lemma}\cite[Lemma 3.3.1]{Wat23}\label{lem:equivbased}
Let $(M,B_M)$,$(N,B_N)$ be based $\uqb$-modules, and let $f : M \to N$ a
$\uqb$-module homomorphism. Then, $f$ is a based $\uqb$-module homomorphism if and only if it satisfies the following:
\begin{enumerate}
	\item $f(\mathcal L_M) \subset  \mathcal L_N$; it thus induces an map $\phi: \mathcal L_M/q^{-1}\mathcal L_M\to \mathcal L_N/q^{-1}\mathcal L_N$ defined by 
	\begin{equation}\label{eq:L}
		b+q^{-1}\mathcal L_M\mapsto f(b)+q^{-1}\mathcal L_N,\qquad \text{where}\qquad b\in B_M.\end{equation}
	\item  $f(_\mathcal AM) \subset  {_\mathcal AN}.$
	\item $f \circ \psi^\imath_M = \psi^\imath_N\circ f$, where $\psi^\imath_M$, $\psi^\imath_N$ denote the $\imath$bar involutions on $M$ and $N$, respectively.
	\item $\phi$ is injective on $\{b \in  \mathcal B_M\,:\, \phi(b) \neq 0\}$.
\end{enumerate}
\end{lemma}

In analogy to Definition \ref{def:dualcan}, we define dual canonical basis of a based $\uqb$-module and their integral form. To each based $\uqb$-module $(M,B)$ we associate $B^\ast$, the dual canonical basis.

\subsection{Characters of $\uqb$}
Based on \cite{Mee24} this section reviews the theory of characters for $\uqb$.
\begin{definition}[Integrable character]\label{def:integrable}
A character $\chi: \uqb\to \Q(q)$ is said to be \emph{integrable} if it occurs in an integrable $\uq$-module.
\end{definition}

\begin{definition}[Hermitian type]\cite[\S4]{Mee24}
A quantum symmetric pair $(\uq,\uqb)$ is said to be of \emph{Hermitian type} if its classical analog $(\mathfrak{g},\mathfrak{g}^\Theta)$ is of Hermitian type. Equivalently, if the corresponding admissible pair $(I_\bullet,\tau)$ is one of Table \ref{table:1}.
\end{definition}

Classically nontrivial characters of $\mathfrak{k}$ occur precisely when $(\mathfrak{g},\mathfrak{k})$ is of Hermitian type, cf. \cite[\S 5.1]{Heckman1995}. Let $i\in I_\circ$ and recall the associated rank one admissible pair as introduced in Definition \ref{def:rankone}.

\begin{notation}
Fix a linear order $(I,\succ)$.
Let $(\uq,\uqb)$ be a Hermitian quantum symmetric pair. Then there exists a unique $i\in I_\circ$ with $i\succeq \tau i$ such that the  irreducible rank one admissible pair associated to $i$ is one of Table \ref{table:1}. We denote $I_{\scriptscriptstyle\otimes }\subset I_\circ$ the corresponding orbit under $\tau$. 
\end{notation}

{\allowdisplaybreaks\begin{table}[H]
		\centering
		\begin{tabular}{|c | c | c |}
			\hline
			Type	& Satake diagram & Restricted root system  \\
			\hline
			$\mathsf{AIII_a}$	&$\begin{dynkinDiagram} A{IIIa}\dynkinRootMark{t}3 \dynkinRootMark{t}8
			\end{dynkinDiagram}$& $\mathsf{BC_n}$  \\
			\hline
			$\mathsf{AIII_b}$	& $\begin{dynkinDiagram} A{IIIb} \dynkinRootMark{t}4
			\end{dynkinDiagram}$ & $\mathsf{C_n}$   \\
			\hline
			$\mathsf{BI}	$&  $\begin{dynkinDiagram} B{o2*.**}\dynkinRootMark{t}1\end{dynkinDiagram}$& $\mathsf{B_2}$   \\
			\hline
			$\mathsf{CI}$	& $\begin{dynkinDiagram} CI \dynkinRootMark{t}4\end{dynkinDiagram}$ &  $\mathsf{C_m}$, $m\geq2$\\
			\hline
			$\mathsf {DI}$	& $\begin{dynkinDiagram} D{o2*.***}\dynkinRootMark{t}1\end{dynkinDiagram}$ &  $\mathsf{B_2}$ \\
			\hline
			$\mathsf{DIII_b}$	& $\begin{dynkinDiagram} D{IIIb} \dynkinRootMark{t}6\dynkinRootMark{t}7\end{dynkinDiagram}$ & $\mathsf{BC_m}$, $m\geq 2$   \\
			\hline		
			$\mathsf{EIII}$	&  $\begin{dynkinDiagram} E{III}\dynkinRootMark{t}1\dynkinRootMark{t}6\end{dynkinDiagram}$& $\mathsf{BC_2}$   \\
			\hline
			$\mathsf{EVII}$	&  $\begin{dynkinDiagram} E{VII} \dynkinRootMark{t}7\end{dynkinDiagram}$&  $\mathsf{C_3}$   \\
			\hline
		\end{tabular}
		\caption{Satake diagrams of Hermitian type and $I_{\scriptscriptstyle\otimes}$} \label{table:1}
\end{table}}

\begin{remark}\label{rem:existence}
If $\chi$ is an nontrivial integrable character for $\uqb$, then the scalars can be extended to $k=\overline{\C(q)}$ to obtain a nontrivial integrable character $\chi^k: \uqb_{\Q(q)}\otimes k\to k$. By \cite[Proposition 4.21]{Mee24} this implies that $\uqb$ is of Hermitian type. For the rest of the paper we thus assume $(\uq,\uqb)$ to be of Hermitian type.
\end{remark}

\begin{notation}\label{rem:n}
It follows from Table \ref{table:1} that the rank one Hermitian quantum symmetric pairs are of type $\mathsf{AIII_a}$ and $\mathsf{AIII_b}$. By Remark \ref{rem:par} these types admit an extra parameter depending on an integer $n$. We use the notation $\uqb=\uqb_n$ to highlight the parameter dependence.
\end{notation}

\begin{notation}
Let $\chi:\uqb\to \Q(q)$ be a character. Denote $V_{\chi}=\mathrm{span}_{\Q(q)}\{1\}$ the $\uqb$-module with
\[b1=\chi(b)1,\qquad \text{where}\qquad b\in \uqb.\]
\end{notation}

By \cite[Proposition 3.2.3]{Wat24}, the algebra anti-automorphism $\varrho$ from \eqref{eq:varrho} restricts to an algebra anti-automorphism of $\uqb$. This is used to conclude Proposition \ref{prop:complement}.

\begin{proposition}\label{prop:complement}\cite[Corollary 3.2.4]{Wat24}
Let $\lambda\in X^+$. Then the highest weight module $L(\lambda)$ regarded as a $\uqb$-module, is completely reducible.
\end{proposition}

\begin{definition}[Spherical vector]
Let $M$ be an integrable $\uq$-module. Then a \emph{spherical vector} $v\in M$ is a nonzero vector spanning a one-dimensional $\uqb$-module.
\end{definition}

\begin{notation}
Let $v\in L(\lambda)$. Then we write $v=v_\mu+\mathrm{l.o.t}$ to indicate that $\mathrm{l.o.t}\in \bigoplus_{\nu<\mu}L(\lambda)_\nu$. 
\end{notation}

\begin{lemma}\label{lem:sometechnical}
Let $\lambda\in X^+$ and let $v\in L(\lambda)$ be a spherical vector. Then
\begin{enumerate}
	\item up to a nonzero scalar multiple, $v=v_\lambda+\mathrm{l.o.t.}$;
	\item the vectors $v$ and $\Upsilon\overline{v}$ are linearly dependent.
\end{enumerate}
\end{lemma}

\begin{proof}
These are \cite[Proposition 4.1 \& Lemma 4.7]{Mee24}. The field of definition in the cited article is $k=\overline{\C(q)}$. However, the same proofs works over $\Q(q)$, mutatis mutandis.
\end{proof}

\begin{definition}[Shift of basepoint]
Let $\chi: \uqb\to \Q(q)$ be an integrable character. Then, the \emph{shift of basepoint} $\rho_\chi:\uqb\to \uq$ is defined by
\[b\mapsto \chi(b_{(1)})b_{(2)},\qquad \text{where}\qquad b\in \uqb.\]
By \cite[Proposition 10.5]{Ko23} $\rho_\chi: \uqb=\uqbs\to \uqds$ is an isomorphism of right $\uq$-comodule algebras, where $d_i=c_i \chi(K_{\tau i}K_i^{-1})$ and $s_i=\chi(B_i)$ for $i\in I_\circ$.
\end{definition}

\begin{lemma}\label{lem:shift}
Let $\lambda,\mu\in X^+$. Let $v\in L(\lambda)$ be a $\chi$-spherical vector for $\uqb$ and let $v'\in L(\mu)$ be a $\eta$-spherical vector for $\rho_{\chi}(\uqb)$. Then $v\otimes v'\in L(\lambda)\otimes L(\mu)$ is a $\eta\circ \rho_\chi$-spherical vector for $\uqb$. Moreover, let $\pi_{\lambda+\mu}: L(\lambda)\otimes L(\mu)\to L(\lambda+\mu)$ denote the Cartan projection. Then $\pi_{\lambda+\nu}(v\otimes v')\neq 0$.
\end{lemma}
\begin{proof}
Let $b\in \uqb$. Then
\[b (v\otimes v')=b_{(1)}v\otimes b_{(2)}v'=\chi(b_{(1)})v\otimes b_{(2)}v'=v\otimes \chi(b_{(1)})b_{(2)}v'=v\otimes \rho_{\chi}(b)v'=\eta\circ \rho_\chi(b) v\otimes v'.\]
Thus, $v\otimes v'$ is a $\eta\circ \rho_\chi$-spherical vector for $\uqb$. Up to a scalar multiple, $v\otimes v'=v_\lambda\otimes v_\mu+\mathrm{l.o.t.}$. As $L(\lambda+\mu)\subset L(\lambda)\otimes L(\mu)$ is the $\uq$-module generated by $v_\lambda\otimes v_\mu$, it follows that the Cartan projection $\pi_{\lambda+\nu}(v\otimes v')$ is nonzero 
\end{proof}

\section{The case of rank one}\label{sec:rankone}
Recall from Table \ref{table:1} that the rank one quantum symmetric pairs of type Hermitian type are of type $\mathsf{AIV}=\mathsf{AIII_a}$ and $\mathsf{AI}=\mathsf{AIII_b}$. In this section we let $(\uq,\uqb)$ be a Hermitian quantum symmetric pair of rank one.
As in \cite{Wat23}, we give parameters so that the vectors spanning one-dimensional modules have crystal limits equal to the highest weight vector.

\subsection{The case $\mathsf{AI}$}\label{sec:AI}
In this subsection we consider the quantum symmetric pair $(\uq,\uqb_n)$ of type $\mathsf{AI}$. Recall from Notation \ref{rem:n} and Remark \ref{rem:par} that we take the parameters $c=q^{-1}$ and $s=[n]=\frac{q^n-q^{-n}}{q-q^{-1}}$, depending on an integer $n$. 
Recall that each integrable $\uq$-module is completely reducible as $\uqb_n$-module, cf. Proposition \ref{prop:complement}. 
The possible eigenvalues of the generator \[B=F+q^{-1}EK^{-1}+[n]K^{-1}\]
on integrable $\uq$-modules are of the form $\{[m]: m\in \Z\}$, cf. \cite[Lemma 4.5]{Mee24}. 
We label the integrable characters of $\uqb_n$ by $\Z$ as follows
\begin{equation}\label{eq:labelchar}
	\chi_l:\uqb_n\to \Q(q),\qquad \qquad \chi_l(B)=[l-n],\qquad l\in \Z.\end{equation}
Denote by $L(n)$ the highest weight $\uq$-module of weight $n\omega$.
The branching rules $[L(n)|_{\uqb_n}: \chi_l]$ are described by:
\begin{equation}\label{eq:branching}
	[L(n)|_{\mathbf{B}_n}: \chi_l]=\begin{cases}
		1&\text{if }|l|\leq n\text{ and } l\equiv n\mod 2\\
		0&\text{else}
	\end{cases}.
\end{equation}

We begin by investigating the \emph{bottom} vectors, i.e. those vectors $f_{l,n}\in L(|l|)$ that span a $\uqb_n$-module of type $\chi_l$. Consider the two dimensional $\uq$-module $L(1)$ with canonical basis $B(1)=\{v_1,v_2\}$. The direct computation of \cite[Lemma 5.1]{Mee25} shows that $f_{1,n}=v_1+q^{-n}v_2$ spans a vector of type $\chi_1$ in $L(1)$ and that $f_{-1,n}=v_1-q^{n}v_2$ spans a vector of type $\chi_{-1}$ in $L(1)$. 

Let $l>0$ and $e=\pm1 $. Using Lemma \ref{lem:shift}, the vector
\begin{equation}\label{eq:a1vector}
f_{el,n}:=f_{e,n}\otimes f_{e,n+e}\otimes \dots\otimes f_{e,n+e(l-1)}\in L(1)^{\otimes l}=L(l)\oplus\bigoplus _{j<l}m_jL(j)
\end{equation}
spans a vector of type $\chi_{el}$. By the branching rules of \eqref{eq:branching}, it follows that \[f_{e,n}\otimes f_{e,n+e}\otimes \dots\otimes f_{e,n+e(l-1)}\in L(l)\subset L(1)^{\otimes l}.\] 

\subsection{The case $\mathsf{AIV}$}\label{sec:AIII_a}
The approach for the quantum symmetric pair $(\uq,\uqb_n)$ of type $\mathsf{AIV^m}$ is analogous to that of $\mathsf{AI}$. Using Notation \ref{rem:n}, the parameters are described by $s=0$, $c_{1,n}=q^n$ and $c_{m,n}=(-1)^mq^{m-1-n}$, depending on an integer $n$.
By \cite[Lemma 4.16]{Mee24}, the integrable characters of $\uqb_n$ are labeled by $\Z$;
\begin{equation}\label{eq:chara}
\chi_l:\uqb_n\to \Q(q),\quad \chi_l(K_1K_m^{-1})=q^l,\quad \chi(B_i)=0,\quad \chi|_{\uq_\bullet}=\epsilon|_{\uq_\bullet},\quad \text{where}\quad i\in I_\circ,l\in \Z.
\end{equation}
Let $l>0$. The branching rules $[L(\lambda)|_{\uqb_n}: \chi_l]$ are described by
\begin{equation}\label{eq:branching2}
	[L(\lambda)|_{\mathbf{B}_n}: \chi_l]=\begin{cases}
		1&\text{if } \lambda=l\omega_1 +\mu \text{ and }\overline{\mu}=0\\
		0&\text{else}.
	\end{cases}.
\end{equation}
\begin{equation}\label{eq:branching3}
	[L(\lambda)|_{\mathbf{B}_n}: \chi_{-l}]=\begin{cases}
		1&\text{if }\lambda=-l\omega_n +\mu \text{ and }\overline{\mu}=0\\
		0&\text{else}.
	\end{cases}.\end{equation}

Consider \eqref{eq:branching2} and \eqref{eq:branching3}, then $f_1=v_1-c_{1,n}^{-1}v_m \in L(\omega_1)$ is a $\chi_1$-spherical vector and $f_{-1}=  v_1+(-1)^mc_{n,m}^{-1}v_m\in L(\omega_1)^\tau$ is a $\chi_{-1}$ spherical vector, cf. \cite[\S 5.1.2]{Mee25}. Here, $v_m=F_m\dots F_2F_1 v_1\in B(\omega_1)$ denotes the lowest weight vector.
Let $l>0$, then Lemma \ref{lem:shift} implies that 
\begin{equation}\label{eq:AIII_aector}
	f_{l,n}:=f_{1,n}\otimes f_{1,n+1}\otimes \dots\otimes f_{1,n+l-1}\in L(\omega_1)^{\otimes l}=L(l\omega_1)\oplus\bigoplus _{\mu< l\omega_1}m_\mu L(\mu)
\end{equation}
is a $\chi_{l}$-spherical vector. By the branching rules of \eqref{eq:branching2}, it follows that 
\[f_{1,n}\otimes f_{1,n+1}\otimes \dots\otimes f_{1,n+l-1}\in L(l\omega_1).\]
We have an analogous description for the vector $f_{-l,n}\in  L(l \omega_m)$.

\subsection{The behavior of spherical vectors at $q=\infty$}\label{sec:regular}
By the explicit computations of subsection \ref{sec:AI}-\ref{sec:AIII_a}, it follows that evaluation at $q=\infty$ of spherical vectors depends on the parameter and on the character. In this subsection we use this information to study based morphisms in rank one.

Let $l\in\Z$ and set $\mu_l\in X^+$ equal to
\begin{equation}\label{eq:mul}
	\mu_l=|l|\omega_1\qquad \text{in type }\mathsf{AI},\qquad\text {and}\qquad
	\mu_l=\begin{cases}
		l\omega_1&\text{if }l\geq0\\
		-l\omega_n&\text{if }l<0
	\end{cases}
	\qquad \text{in type }\mathsf{AIV}.
\end{equation}

\begin{notation}\label{not:chil}
Let $l\in \Z$. Denote $\chi_l:\uqb_n\to \Q(q)$ the character of subsections \ref{sec:AI}-\ref{sec:AIII_a}. Moreover, denote $f_{n,l}=v_\lambda+\mathrm{l.o.t}\in L(\mu_l)$ the $\chi_l$-spherical vector for $\uqb_n$. The existence of $f_{n,l}$ is assured by subsections \ref{sec:AI}-\ref{sec:AIII_a}.
\end{notation}

Let $l\geq 0$ and $e=\pm1$. Using Notation \ref{not:cartan} we denote $\pi_{\mu_{el}}: L(\mu_{e})^{\otimes l}\to L(\mu_{el})$ the repeated application of Cartan projections. 

\begin{lemma}\label{lem:injective}
The spherical vector $ f_{l,n}\in L(\mu_l)$ of $\uqb_n$
satisfies $\pi_{\mu_l}(f_{l,n})\equiv_\infty v_{\mu_l}$ if and only if
\begin{enumerate}
	\item In case $\mathsf{AI}$,
	$n> 0$ in case $l> 0$ and $0> n$ in case $l<0$.
	\item In case $\mathsf{AIV^m}$, 
	$n> 0$ in case $l> 0$ and $m-1>n$ in case $l<0$.
\end{enumerate}
\end{lemma}

\begin{proof}
Repeated application of Lemma \ref{lem:preserveaform} shows that $\pi_{\mu_l}(f_{l,n})\equiv_\infty v_{\mu_l}$ if and only if $f_{1,n+k}\equiv_{\infty} v_{\mu_1}$ for all $0\leq k \leq l-1$. The conditions on $n$ then follow by the explicit description of the spherical vectors, as determined in subsections \ref{sec:AI}-\ref{sec:AIII_a}.
\end{proof}

\begin{lemma}\label{lem:basedrankone}
The $\uqb_n$-module $V_{\chi_l}$ carries a up to scalar multiple unique structure of based $\uqb_n$-module.
\end{lemma}

\begin{proof}
Fixing a non-zero vector $v\in V_{\chi_l}$ uniquely determines the conditions $(i)$, $(iii)$ and $(iv)$ in Definition \ref{def:Bbased}. By Lemma \ref{lem:sometechnical}, $v$ is a weight vector of weight $\overline{\mu_l}$, making $V_{\chi_l}$ a weight $\uqb_n$-module.
We remain to verify that 
\begin{equation}\label{eq:preserve2}
	_\mathcal A \dot{\uqb}_n v\subset \mathcal Av.\end{equation}
We recall by \eqref{eq:genaform} that, as an $\mathcal A$-algebra, $_\mathcal A \dot{\uqb}_n$ is generated by the elements $B_{i,\zeta}^{(a)}$ and $E_j \boldsymbol{1}_{\zeta}$ for $\zeta\in X_\imath$, $i\in I_\circ $, $j\in I_\bullet$ and $a\geq0$. Thus it is sufficient to check that these generators preserve \eqref{eq:preserve2}. We have $\boldsymbol{1}_\zeta v=0$ if $\zeta\neq \overline{\mu_l}$. Now, consider the case $\zeta=\overline{\mu_l}$.
Let $i\in I_\circ$, $j\in I_\bullet$ and $a\geq 0$. Indeed \eqref{eq:preserve2} holds for $E_j 1_{\overline{\mu_l}}$ as $E_j v=0$ by \eqref{eq:chara}. Moreover, by Definition \ref{def:idiv}, the action of $B_{i,\overline{\mu_l}}^{(a)}$ on $v$ equals the action of an polynomial in $B_i$ with coefficients in $\mathcal A$. By \eqref{eq:chara} and \eqref{eq:labelchar}, $\chi_l(B_i)$ has its values in $\mathcal A$. It thus follows that $B_{i,\overline{\mu_l}}^{(a)}v\in \mathcal Av$, which concludes that \eqref{eq:preserve2} holds.
\end{proof}

In rank one, we recall that the monoid of spherical weights $\breve{X}^+$ has a single generator $\varpi$. 

\begin{lemma}\cite[Lemma 4.1.2, Lemma 4.1.1]{Wat23}\label{lem:wat}
	There exists $w_0 \in  L(\varpi)$ such that $\text{span}_{\Q(q)} \{w_0\} \cong L(0)$ as $\uqb_n$-modules and $w_0\equiv_\infty v_\varpi$.
\end{lemma} 

Note that the existence of the \emph{$K$-matrix} is assured by the choice of parameters from Remark \ref{rem:par} and \cite[(5.17) \& Corollary 7.7]{Ko19}. Thus, the proof of \cite[Lemma 4.1.2, Lemma 4.1.1]{Wat23} remains valid for the parameters of Remark \ref{rem:par}.

Recall the weight $\mu_l$ from \eqref{eq:mul}.

\begin{theorem}\label{thm:rankone}
Let $n\in \Z$ so that the conclusion of Lemma \ref{lem:injective} holds. Then, for each $k \geq 0$ there exists a unique 
morphism $g_k : L(\mu_l+k\varpi) \to V_{\chi_l}$ of based $\uqb_n$-modules such that $g_k(v_{\mu_l+k\varpi}) =1.$
\end{theorem}

\begin{remark}
The proof of Theorem \ref{thm:rankone} follows the approach of \cite[Proposition 4.6.1]{Wat23}.
\end{remark}

\begin{proof}
Recall by Lemma \ref{lem:basedrankone} that $V_{\chi_l}$ carries a up to scalar multiple unique structure of based $\uqb_n$-module. Fix the based $\uqb_n$-module $(V_{\chi_l},\{1\})$.

We first consider the induction base $k=0$. 
Consider the vector $v= f_{l,n}$ from Definition \ref{not:chil}. Using Lemma \ref{lem:sometechnical}, set 
$W_0= \mathrm{span}_{\Q(q)}\{v\}$ and denote $W_1\subset L(\mu_l)$ the complement of $W_0$ with respect to the inner product of subsection \ref{sec:base}. As $(v_\lambda,v)=1$, we have $v_\lambda=cv+w_1$ with $w_1\in W_1$ and $c=\frac{1}{(v,v)}\in 1+\mathbf{A}_\infty$.
Proposition \ref{prop:complement} implies that the map $g_0$ arises as
\[g(v)=c^{-1},\quad g(W_1)=\{0\}.\]
It thus follows that $g_0$ preserves crystal lattices, and we have $g(v_\lambda)=1$. This moreover shows that $g_0$ preserves $\mathcal A$-forms and that the induced map \eqref{eq:L} satisfies 
\[b+q^{-1}\mathcal L\mapsto \delta_{b,v_\lambda}+q^{-1}\mathcal L,\qquad b\in B(\mu_l).\]
Recall from Lemma \ref{lem:sometechnical} that $\Upsilon \overline{v}=av$, for a nonzero scalar $a$. As the complement $\mathrm{span}_{\Q(q)}\{v\}^\perp$ is $\imath$bar invariant, we conclude that
\[g_0\circ \psi^\imath_{L(\mu_l)}=a'\, \psi^\imath_{V_{\chi_l}}\circ g,\qquad \text{where}\qquad a'\in \Q(q).\]
Because \[g_0(\psi^\imath(v_\lambda))=g_0(v_\lambda)=(v_\lambda,v)=1=\psi^\imath(1),\]
it follows that $a'=1$.
Thus, we may apply Lemma \ref{lem:equivbased} to conclude that $g_0$ is based.

With use of Lemma \ref{lem:shift}, the induction step is shown analogous to \cite[Proposition 4.6.1]{Wat23}.
\end{proof}

\section{Integrable characters over $\Q(q)$}\label{Sec:Char}
In this section we classify the integrable characters and relate their branching rules to those over the bigger field $k=\overline{\C(q)},$ the algebraic closure of $\C(q)$. In this section we let $(\uq,\uqb_n)$ be a Hermitian quantum symmetric pair of general rank. We recall from Notation \ref{rem:n}  that the parameters of $\uqb_n$ depend on an integer $n$. Furthermore, recall the restricted root system from \eqref{eq:restroot} and Table \ref{table:1}.
\begin{proposition}\label{prop:fieldextend}
	For each $l\in \Z$, there exists a character $\chi_l: \uqb_n\to \Q(q)$ with:
	\begin{enumerate}
		\item For all $ i\in I_{\scriptscriptstyle\otimes},j\in I_\circ\setminus I_{\scriptscriptstyle\otimes}$
		\[
		\chi_l(K_jK_{\tau j}^{-1})=1,\quad \chi_l(B_i)=[l-n]_{i},\quad \chi_l(B_j)=0,\quad \chi_l|_{\uq_\bullet}=\epsilon
		\]
		in case the restricted root system $\Sigma$ is reduced.
		\item For all $i \in I_{\scriptscriptstyle\otimes},j\in I_\circ\setminus I_{\scriptscriptstyle\otimes}$ with $i\succeq \tau i$
		\[\chi_l(K_iK_{\tau i}^{-1})=q_i^l,\quad \chi_l(K_jK_{\tau j}^{-1})=1,\quad \chi_l(B_j)=0,\quad \chi_l|_{\uq_\bullet}=\epsilon\]
		if the restricted root system $\Sigma$ is not reduced. 
	\end{enumerate}
Moreover, the characters $\chi_l$ are pairwise inequivalent and provide a classification of the integrable characters.
\end{proposition}

\begin{proof}
	The existence of the characters $\chi_l^k:\uqb_n\otimes_{\Q(q)} k\to k$ with the required defining relations follow by \cite[Lemma 4.16]{Mee24} and \cite[Theorem 7.2]{Schlichtkrull_1984}. The character $\chi_l^k$ maps all the generators of $\uqb_n$ to scalars in $\Q(q)$, thus it follows that the character $\chi_l^k$ can be restricted to a character $\chi_l:\uqb_n\to \Q(q)$. This yields the desired character. Moreover, by \cite[Lemma 4.16]{Mee24} and Remark \ref{rem:existence} it follows that all integrable characters are of the described form. 
\end{proof}

We shall hereafter suppress $n$ from the notation, and write simply $\uqb$ for the coideal subalgebra $\uqb_n$.

The next goal is to show that the characters of Proposition \ref{prop:fieldextend} are integrable in the sense of Definition \ref{def:integrable}. We tackle this using Watanabe's framework of \textit{integrable $\uqb$-modules} \cite{Wat2025}.

\begin{notation}
As in the proof of Proposition \ref{prop:fieldextend}, write $\chi_l^k: \uqb\otimes_{\Q(q)}k\to k$ for the character having the same defining relations as $\chi_l$. 
\end{notation}

For all $\lambda,\mu\in X^+$ let $V^\imath(\lambda,\mu)\subset L(\lambda)\otimes L(\mu)$ denote the $\uq$-module generated by $v_{\lambda,\mu}=v_{w_\bullet \lambda}\otimes v_\mu$, here $v_{w_\bullet \lambda}\in B(\lambda)_{w_\bullet \lambda}$. Denote $V^\imath(\lambda)$ the module $V^\imath(0,\lambda)$ that as a $\uq$-module is isomorphic to $L(\lambda)$. 

\begin{definition}\cite[Definition 3.3.4]{Wat2025}.
	A weight $\uqb$-module $V$ is said to be \emph{integrable}, if for each weight vector $v\in V$ there exist weights $\lambda,\mu \in X^+$  and a $\uqb$-module homomorphism $V^\imath(\lambda,\mu) \to V$ which sends
	$v_{\lambda,\mu}$ to $v$.
\end{definition}

Note that by definition integrable $\uqb$-modules are local quotients of integrable $\uq$-modules. 

\begin{notation}\label{notaion:bottom}
Let $l\in \Z$ and let $i\in I_{\scriptscriptstyle\otimes }$ with $i\succeq \tau i$. Set $\mu_l\in X^+$ equal to
\begin{equation}\label{eq:mul2}
	\mu_l=|l|\omega_i\qquad \text{if }\Sigma\text{ is reduced},\qquad
	\mu_l=\begin{cases}
		l\omega_i&\text{if }l\geq0\\
		-l\omega_{\tau i}&\text{if }l<0
	\end{cases}
	\qquad \text{if }\Sigma\text{ is not reduced}.
\end{equation}
\end{notation}

\begin{remark}
	Note that the Notation \ref{notaion:bottom} generalizes the definition of $\mu_l$, as given in \eqref{eq:mul}
\end{remark}

\begin{lemma}\label{lem:integrable}
	The characters $\chi_l$ from Proposition \ref{prop:fieldextend} define integrable $\uqb$-modules.
\end{lemma}

\begin{proof}
From Proposition \ref{prop:fieldextend} it follows that $V_{\chi_l}$ is a weight $\uqb$-module of weight $\overline{\mu_l}$. In the case that the restricted root system $\Sigma$ is not reduced it follows from Proposition \ref{prop:fieldextend} that 
	\begin{equation}\label{eq:intreduced}
	E_jv=F_jv=B_iv=0,\qquad \text{for all}\qquad i\in I_\circ,j\in I_\bullet.
	\end{equation}
Thus it follows from \cite[Proposition 4.3.1]{Wat2025} that $V_{\chi_l}$ is integrable. Consider the case that $\Sigma$ is reduced. Then it follows from Proposition \ref{prop:fieldextend} that
	\begin{equation}\label{eq:intnotred}
	E_jv=F_jv=0,\qquad \text{for all}\qquad j\in I_\bullet.
	\end{equation}
	Let $i\in I_{\scriptscriptstyle\otimes}$ and let $v\in V_{\chi_l}$. Consider the rank one QSP coideal subalgebra $\uqb_i$ from Definition \ref{def:rank}. The character $\chi_l$ restricted the coideal subalgebra $\uqb_i$ is the character described in \eqref{eq:labelchar}. Thus by \cite[Proposition 4.1.1]{Wat2025}, we have
	\begin{equation}\label{eq:integrable}
		B_{i,\overline{\mu_l}}^{(l+1)}v=0,\end{equation}
	where $B_{i,\overline{\mu_l}}^{(l+1)}v$ denotes the $\imath$divided power from Definition \ref{def:idiv}. Therefore, if the restricted root system $\Sigma$ is reduced, \cite[Proposition 4.3.1]{Wat2025} can also be applied to conclude that the $\uqb$-module $V_{\chi_l}$ is integrable.
\end{proof}

\begin{corollary}
The characters $\chi_l$ are integrable in the sense of Definition \ref{def:integrable}.
\end{corollary}

\begin{theorem}\label{thm:branching}
	The branching rules of $\chi_l$ coincide with those of $\chi_l^k$, that is to say
	\[[L(\lambda)|_\uqb: \chi_l]=[L(\lambda)\otimes _{\Q(q)}k|_{\uqb\otimes _{\Q(q)}k}: \chi^k_l]=\begin{cases}
	1&\text{if\,\,}\lambda=\mu_l+\mu\text{ with }\mu\in \breve{X}^+\\
	0&\text{else}
	\end{cases}.\]
Moreover, all $\uqb$-homomorphisms $\pi:L(\mu_l+\lambda) \to V_{\chi_l}$ are a scalar multiple of the $\uqb$-homomorphism
$$L(\mu_l+\lambda) \to V_{\chi_l},\qquad\text{with} \qquad  v_{\lambda+\mu_l}\mapsto 1.$$
\end{theorem}

\begin{proof}
By Remark \ref{rem:existence} it follows that
\[[L(\lambda)|_\uqb: \chi_l]\geq [L(\lambda)\otimes _{\Q(q)}k|_{\uqb\otimes _{\Q(q)}k}: \chi^k_l],\qquad \text{where}\qquad\lambda\in X^+.\]
An argument analogous to \cite[Lemma 4.2]{Mee24} shows that $ [L(\lambda)|_\uqb: \chi_l]\leq1.$
Thus to conclude that the branching rules coincide, it is sufficient to show that $[L(\lambda)\otimes _{\Q(q)}k|_{\uqb\otimes _{\Q(q)}k}: \chi^k_l]=1$ implies that $[L(\lambda)|_\uqb: \chi_l]=1$. 

First we consider the case that $\Sigma$ is reduced. Recall from \cite[Corollary 4.20]{Mee24} that 
\[[L(\lambda)\otimes _{\Q(q)}k|_{\uqb\otimes _{\Q(q)}k}: \chi^k_l]=1\text{ if and only if }\lambda =\mu_l+\nu\text{ with }\overline{\nu}=0.\]
Here, we recall that $\overline{\nu}=0$ if and only if $\nu\in \breve{X}^+$.
According to Notation \ref{notaion:bottom}, we have 
\[\langle \alpha_i^\vee, \mu_l+\nu\rangle \geq\langle \alpha_i^\vee, \mu_l\rangle =\langle \alpha_i^\vee, |l| \omega_i\rangle= |l|.\]
Combining this with \eqref{eq:intreduced}, we apply \cite[Proposition 4.1.1 \& Theorem 4.2.6]{Wat2025} to obtain the existence of a $\uqb$-module homomorphism $L(\lambda)\to V_{\chi_l}$ with $v_\lambda\mapsto 1$.

Lastly, consider the case that $\Sigma$ is not reduced. Let $v\in L(\lambda)\otimes_{\Q(q)}k$ span a one-dimensional module of type $\chi_l^k$. By \cite[Proposition 4.1]{Mee24}, we have $v=v_{\lambda}+l.o.t.$ and therefore $K_h v=q^{\langle h,\lambda\rangle} v$ for all $h\in Y^\imath$. In particular the $\imath$weight of $\overline{\lambda}$ coincides with that of $V_{\chi_l}$.  Thus, using Proposition \ref{prop:fieldextend} we apply \cite[Proposition 4.1.1 \& Theorem 4.2.6]{Wat2025} to obtain the existence a $\uqb$-module homomorphism $V(\lambda)\to V_{\chi_l}$ with $v_\lambda \mapsto 1$.
\end{proof}

\begin{lemma}
	The $\uqb$-module $V_{\chi_l}$ carries a up to scalar multiple unique structure of based $\uqb$-module.
\end{lemma}

\begin{proof}
By Lemma \ref{lem:integrable} $V_{\chi_l}= \mathrm{span}_{\Q(q)}\{1\}$ is a weight $\uqb$-module. As in Lemma \ref{lem:basedrankone}, we remain to verify the condition 
\begin{equation}\label{eq:preserve}
	_\mathcal A \dot{\uqb} 1\subset \mathcal A1.\end{equation}
It suffices to check that the generators $B_{i,\zeta}^{(a)}$ and $E_j \boldsymbol{1}_{\zeta}$ for $\zeta\in X_\imath$, $i\in I_\circ $, $j\in I_\bullet$ and $a\geq0$, preserve \eqref{eq:preserve}. By Proposition \ref{prop:fieldextend} $\chi_l|_{\uq_\bullet}=\epsilon$, thus the generators $E_j \boldsymbol{1}_{\zeta}$ for $j\in I_\bullet$ annihilate $V$. Now, it is sufficient to check \eqref{eq:preserve} for $i\in I_\circ $ with $\chi_l(B_i)\neq 0$. There is at most one $i\in I_\circ$ with $\chi(B_i)\neq 0$, namely $i\in I_{\mathrm{ns}}\cap I_{\scriptscriptstyle\otimes }$. We therefore reduce to the rank one case, which is shown in the proof of Lemma \ref{lem:basedrankone}.
\end{proof}

\begin{proposition}\label{prop:aform}
	Let $\pi:L(\mu_l+\lambda) \to V_{\chi_l}$ be the $\uqb$-homomorphism as described in Theorem \ref{thm:branching}. Then $\pi({_\mathcal A}L(\lambda))=\mathcal A 1.$
\end{proposition}

\begin{proof}
	By \cite[Theorem 5.3 (Claim $\star$)]{Bao18} we have $_\mathcal AL(\lambda)={_\mathcal A\dot{\uqb}}v_\lambda$, thus by \eqref{eq:genaform} and \eqref{eq:preserve} we have
	\[\pi( _\mathcal A L(\mu_l+\lambda))=\pi( _\mathcal A \dot{\uqb} v_{\mu_l+\lambda} )=\mathcal A \pi(v_{\mu_l+\lambda})=\mathcal A 1.\qedhere\]
\end{proof}

\section{A reduction to rank one}\label{Sec:Randreduc}
In this section we generalize Theorem \ref{thm:rankone} to a Hermitian quantum symmetric pair of general rank. 

\begin{remark}\label{rem:parameters}
Let $i\in I_{\scriptscriptstyle \otimes}$. For the rest of the paper the parameters $(c_i,s_i)$ of $\uqb$ are assumed to satisfy the conditions of Lemma \ref{lem:injective}. 
\end{remark} 

\begin{lemma}\label{lem:qinfty}
Let $f\in L(\lambda)$ be a $\chi$-spherical vector normalized with $f=v_\lambda+\mathrm{l.o.t}$, then $f\equiv_{\infty} v_\lambda$.
\end{lemma}

\begin{proof}
Let $c\in \Q(q)$ be such that $w_0=cf\in \mathcal  \mathcal L(\lambda) \setminus  q^{-1}\mathcal L(\lambda)$.
Let $i \in I_\bullet$. By \cite[Lemma 4.16]{Mee24} it follows that
\[E_iw_0 = F_iw_0 = 0.\]
Thus
\[\tilde{E}_iw_0 = 0,\]
where $\tilde{E}_i$ denotes the Kashiwara operator.
Let $i \in  I_\circ$ and consider the associated quantum symmetric pair $(\uq_i,\uqb_i)$ of rank one from Definition \ref{def:rank}. Next, consider the $\uq$-module $L(\lambda)$ as a semisimple $\uq_i$-module. Here we can apply the rank one cases of Theorem \ref{thm:rankone} and \cite[Corollary 4.2.3]{Wat24} (dealing with the more general parameters using Lemma \ref{lem:wat}) to conclude that
\[\tilde{E}_iw_0 \equiv_\infty 0.\]
Because $\widetilde{E}_i w_0\equiv_\infty0$ for all $i\in I$, it follows from \cite[Example 2.3.4]{Wat24} that there exists $0\neq a\in \Q$ with $aw_0\equiv_\infty v_\lambda$. As $f=v_\lambda+\mathrm{l.o.t.}$, this shows that $f\equiv_\infty v_\lambda$.
\end{proof}

\begin{theorem}\label{thm:mainthm}
	Let $\lambda\in \breve{X}^+$ and let $l\in \Z$. Then there exists a unique morphism as based $\uqb$-modules
	$$L(\mu_l+\lambda) \to V_{\chi_l},\qquad \qquad  v_{\lambda+\mu_l}\mapsto 1.$$
\end{theorem}

\begin{proof}
The existence of the $\uqb$-homomorphism $g: L(\mu_l+\lambda)\to V_{\chi_l}$ with $v_{\lambda+\mu_l}\mapsto 1$ is concluded from Theorem \ref{thm:branching}. We apply Lemma \ref{lem:equivbased} to conclude that the morphism is based. We must verify required conditions. Part $(i)$ and $(iv)$ in Lemma \ref{lem:equivbased} follow by Lemma \ref{lem:qinfty}. Part $(ii)$ in Lemma \ref{lem:equivbased} follows from Proposition \ref{prop:aform}.
Part $(iii)$ in Lemma \ref{lem:equivbased}, commutation with the $\imath$bar involutions, is proved in an analogous way as in Theorem \ref{thm:rankone}.
\end{proof}

\section{Compatibility with integral forms}\label{Sec:Int}
In this section we show that spherical vectors are compatible with the integral forms of the canonical basis. As a first step we show compatibility with the integral form of the dual canonical basis, these integral forms were introduced in Definition \ref{def:dualcan} and the end of subsection \ref{sec:qsp}.

\begin{lemma}\label{lem:inclusion}
Let $M$ and $N$ be a based $\uqb$-modules and let $\pi:M\to N$ be a linear map with $\pi(_\mathcal A M)\subset _\mathcal A N$. Then $\pi^\ast({_\mathcal A M^\ast })\subset  {_\mathcal A N^\ast}$.
\end{lemma}

\begin{proof}
By Lemma \ref{lem:equivbased}, the map $\pi$ preserves the $\mathcal A$-forms. Thus, the induced map
\[ \,\pi^\ast: \Hom_{\mathcal A}(_\mathcal A N, \mathcal A)\to \Hom(_\mathcal A M, \mathcal A),\qquad f\mapsto  f\circ \pi,\]
gives rise to the inclusion $\pi^\ast({_\mathcal A M^\ast })\subset  {_\mathcal A N^\ast}$.
\end{proof}

\begin{notation}
	Let $\lambda \in \breve{X}^+$ and $l\in \Z$. Denote $f^\lambda_{l,n}= v_{\mu_l+\lambda}+\mathrm{l.o.t}\in L(\mu_l+\lambda)$ the $\chi_l$ spherical vector for $\uqb_n$. Its existence is assured by Theorem \ref{thm:branching} and Lemma \ref{lem:sometechnical}.
\end{notation}

Recall that $_\mathcal A L^{\mathrm{up}}(\lambda)\subset L(\lambda)$ denotes the integral form of the dual canonical basis as introduced in Section \ref{Sec:QG}.

\begin{lemma}\label{lem:dualcan}
Let $\lambda \in X^+$, then $f^\lambda_{l,n}\in {_\mathcal AL^{\mathrm{up}}(\mu_l+\lambda)}$.
\end{lemma}

\begin{proof}
Let $v=f^\lambda_{l,n}$ and $V=\mathrm{span}_{\Q(q)}\{v\}$. Consider the based module $(V,\{v\})$. It follows from Proposition \ref{prop:aform} that the $\uqb_n$-homomorphism with $v_\lambda\mapsto v$ satisfies $\pi( {_\mathcal AL}(\lambda))= \mathcal A \{v\}$. This homomorphism has the explicit realization with respect to the inner product of subsection \ref{sec:base}
\begin{equation}\label{eq:realization}
w\mapsto (w,v)v,\qquad\text{where}\qquad w\in L(\lambda).
\end{equation}
By Lemma \ref{lem:inclusion} this means that $\pi^\ast (v^\ast)\ \in {_\mathcal A L(\lambda)^\ast}$.
When identifying ${_\mathcal A L(\lambda)^\ast}\cong {_\mathcal AL^{\mathrm{up}}(\lambda)}$, the realization \eqref{eq:realization} implies that $\pi^\ast(v^\ast)\mapsto v$. Thus $v\in {_\mathcal AL^{\mathrm{up}}(\lambda)}$.
\end{proof}

\begin{theorem}\label{lem:integral1}
Let $l\in \Z$. Then $f^0_{l,n}\in {_\mathcal AL(\mu_l)}$
\end{theorem}

\begin{proof}
We first show this for $e=\pm1$. 
Let $v=f^0_{e,n}$. By a case-by-case verification using Table \ref{table:1}, the weights $\mu_e$ from Notation \ref{notaion:bottom} are quasi-minuscule. As $\mu_e$ is quasi-minuscule, an analogous computation to  \cite[Example 9.21]{Jan96} (here the bilinear form is contravariant with a slightly different algebra anti-automorphism, however the evaluations of the forms differ by a unit in $\mathcal A$) shows that
\begin{equation}\label{eq:aform}
	\bigoplus _{\lambda\neq 0}{_\mathcal AL^{\mathrm{up}}(\mu_e)}_\lambda =\bigoplus _{\lambda\neq 0}{_\mathcal AL(\mu_e )}_\lambda.\end{equation}
By Notation \ref{notaion:bottom}, we remark $\overline{\mu_{e}}\neq \overline{0}$. As $v$ is an $\imath$weight vector of weight $\overline{\mu_e}$, Lemma \ref{lem:dualcan} and \eqref{eq:aform} imply that $v\in{_\mathcal AL(\lambda)}$. 
Let $l> 0$. Consider the vector
$f_{e,n}^0\otimes f_{e,n+e}^0\otimes \dots\otimes f_{e,n+e(l-1)}^0\in L(\mu_e)^{\otimes l}$. 
Recall that $\pi_{\mu_{el}}: L(\mu_e)^{\otimes l}\to L({\mu_{el}})$ is the Cartan projection. 
We deduce from Lemma \ref{lem:shift} that the vector
\begin{equation}\label{eq:somevector}
	\pi_{\mu_{el}}(f_{e,n}^0\otimes f_{e,n+e}^0\otimes f_{e,n+e(l-1)}^0)=v_{\mu_l}+\mathrm{l.o.t.}
\end{equation}
is $\chi_{el}$-spherical. By the branching rules of Proposition \ref{prop:fieldextend} it follows that \eqref{eq:somevector} equals $f^0_{el,n}$ and that the projection on all the other isotypic components is zero. By \cite[Proposition 27.1.7 \& Theorem 27.3.2]{Lus10} this implies $f^0_{el,n}\in {_\mathcal AL(\mu_l)}$.
\end{proof}
\printbibliography
\Addresses
\end{document}